\documentclass{amsart}

\usepackage{amsfonts, amssymb, amsmath,amscd, url}

\newtheorem{theorem}{Theorem}

\newtheorem{conjecture}{Conjecture}\newtheorem{proposition}{Proposition}
\newtheorem{lemma}{Lemma}

\def\Q{\mathbb Q}

\def\tfr#1{{\widehat{#1}}}

\def\tfrl#1{{\widehat{#1_{\lambda}}}}
\def\Det{\mathfrak D}
\def\walsh#1{{#1}^{{\mathcal W}}}

\def\fd{{\mathbb F}_2}

\def\fdm{{\mathbb F}^m_2}
\def\fp{{\mathbb F}_p}
\def\fdten{{\mathbb F}_{1024}}

\def\nbz{{\rm nbz}\,}

\def\tracabs#1{{\rm Tr}_{#1}}

\def\grmul#1{{{#1}^\times}}

\def\Tr{{\rm Tr}\,}

\title{On a conjecture of Helleseth}
\author{Yves Aubry}
\address{Institut de Math\'ematiques de Toulon, Universit\'e du Sud Toulon-Var, France and Institut de Math\'ematiques de Luminy, Marseille, France}
\email{yves.aubry@univ-tln.fr}

\author{Philippe Langevin}
\address{Institut de Math\'ematiques de Toulon, Universit\'e du Sud Toulon-Var, France}
\email{langevin@univ-tln.fr}

\begin{document}

\begin{abstract} 
We are concern about  a conjecture
proposed in the middle of the seventies by Hellesseth 
in the framework of maximal sequences and theirs 
cross-correlations.  The conjecture claims the existence 
of a zero outphase Fourier coefficient. We give 
some divisibility properties in this direction.
\end{abstract}

\maketitle


\section{Two conjectures of Helleseth}

Let $L$ be a finite field of order $q>2$ and characteristic $p$.
Let $\mu$ be the canonical additive character of $L$ i.e. 
$\mu(x)=\exp( {2i\pi\Tr(x)}/p )$ where $\Tr$ is the trace function with respect to the finite field extension  $L/{\mathbb F}_p$.
The \emph{Fourier coefficient}  of a 
mapping  $f\colon L\rightarrow L$ is defined at $a\in L$ by
\begin{equation}
\tfr f(a) = \sum_{x\in L} \mu( ax + f( x ) ).
\end{equation}

The distribution of these values is called the \emph{Fourier spectrum} 
of $f$. Note that when $f$ is a permutation the \emph{phase} Fourier 
coefficient $\tfr f(0)$ is equal to 0. 

The mapping $f(x) = x^s$ is called the power function of exponent $s$, 
and it is a permutation if and only if $(s,q-1)=1$. 
Moreover, if $s\equiv 1 \mod {(p-1)}$ the Fourier coefficients of $f$ are rational integers. 
Helleseth made in \cite{TOR} two ``global'' conjectures on the spectra of power
permutations.  The first claims the vanishing of the quantity 
(related to Dedekind determinant, see \cite{LANG})
\begin{equation}
\label{DET}
        \Det(f) = \prod_{a\in\grmul L} \tfr f(a).
\end{equation}

\begin{conjecture}[Helleseth]
\label{THC} Let $L$ be a field of cardinal $q>2$. If $f$ is
a power permutation of exponent $s \equiv 1\mod{(p-1)}$ then  $\Det(f) = 0$.
\end{conjecture}

For $p=2$, it generalizes Dillon's conjecture (see \cite{Dillon})
which corresponds to the case $s=q-2\equiv -1\mod{(q-1)}$, and known to be true 
because it is related to the vanishing of Kloosterman sums and the class number $h_q$
of the imaginary quadratic number field $\Q( \sqrt{1-4q})$ (see  \cite{KL, LW}).
Note also that in odd characteristic the Kloosterman sums do not
vanish (see \cite{KRV}) except if $p=3$ (see \cite{KL}). 

The second conjecture deals with the number of values in 
the spectrum of a power permutation.
\begin{conjecture}
\label{POW}
If $[L:\fp]$ is a power of 2 then
the spectrum of a power function takes
at least four values.
\end{conjecture}

In this note, we prove some results concerning the divisibility
properties of the Fourier coefficients of power permutations in 
connection with Conjecture \ref{THC}. Our results can be seen as 
a proof ``modulo $\ell$'' of Conjecture \ref{THC} for certain primes $\ell$. 

\section{Boolean function case}

In this section, we assume $p=2$.
In \cite{NPP}, the second author has computed the Fourier spectra of power permutations 
for all the fields of  characteristic 2 with degree less or equal to $25$ 
without finding any counter-example to the above conjectures. More curiously, if
we denote by $\nbz(s)$ the number of Fourier coefficients of the power function of exponent $s$ equal to zero then the
numerical experience suggests that:

$$
\nbz( s ) \geq \nbz(-1) = h_q.
$$

At this point, it is interesting to notice that Helleseth's conjecture
can not be extended to the set of all permutations. Indeed, let $m$ be a positive
integer and let $g\colon\fdm\rightarrow\fd$ be a Boolean function in $m$ variables.
One defines the Walsh coefficient of $g$ at $a\in\fdm$ by :

$$
\walsh g(a) = \sum_{x\in\fdm} (-1)^{a.x+g(x)}.
$$

Identifying $L$ with the $\fd$-vector space $\fdm$, the Boolean function $g$ has a trace representation 
i.e. there exists a mapping $f\colon L\rightarrow L$ such that $g(x)=\tracabs L(f(x))$
for all $x$ in $L$. Of course, the trace representation is not unique. Moreover, if $g$ is 
balanced then $g$ can be represented by a permutation of $L$. In all the cases, the Walsh 
spectrum of $g$ and the Fourier spectrum of $f$ are identical.

In \cite{EXAMPLE},  an example of a ten variables Boolean 
function with a very atypical Walsh spectrum (see Tab. \ref{BOOLE}) is given. This Boolean function 
is balanced and its Walsh coefficients vanish only once. This numerical example, say $g$, implies the existence  
of a permutation $f$ of $\fdten$  (not a power permutation)  such that 
$$
	g(x) = \tracabs {\fdten} f(x),
$$
whence the Fourier spectrum of $f$ is equal to the Walsh spectrum of $g$, and 
thus
$\sum_{x\in\fdten } \mu( ax + f(x)  ) \not= 0$ 
for all $a\in\fdten^{\times}$.

\begin{table}
\caption{\label{BOOLE} An example of Walsh spectrum having only one 
Walsh coefficient equal to zero (see \cite{EXAMPLE}).}
\begin{tabular}{lcccccccccc}
\hline
Walsh &-48 &-44 &-40 &-36 &-32 &-28 &-24 -&20 &-16 &-12\\
\hline
mult. &5 &30 &85 &70 &115 &100 &31 &62 &20 &10\\
\hline
\hline
Walsh &0 &8 &16 &20 &24 &28 &32 &36 &40 &44\\
\hline
mult. &1 &5 &25 &20 &85 &90 &90 &80 &50 &50\\
\hline
\end{tabular}
\end{table}

A possible generalization of the conjecture of Helleseth,
proposed by Leander, could be the following one:
\begin{conjecture}
\label{GLC}
If $f$ is a permutation of $L$ then
$
\prod_{\lambda\in\grmul L} \Det( \lambda f) = 0.
$
\end{conjecture}

Note that Conjecture \ref{POW} is know to be true in characteristic 2
since recent works of Daniel Katz in \cite{KATZ} and Tao Feng in
\cite{FENG}. In order to complete this short conjecture tour , we recall to the
reader the  main global conjecture of the domain due to Sarwate and 
which is still open

\begin{conjecture}
\label{MAIN}
If $f$ is a power permutation of $L$ where $[L:\fd]$ is even
then  $\sup_{a\in L} \tfr f(a) \geq 2 \sqrt q$.
\end{conjecture}

In the sequel, if $\lambda \in L$ then we denote by $\tfrl f(a)$ the Fourier
coefficient of $x\mapsto \lambda f(x)$. If $f$ is a
power permutation of exponent $s$, denoting 
by $t$ the inverse of $s$ modulo $q-1$, for all
$y\in\grmul L$, we have :

\begin{equation}
\label{SPEC}
\tfrl f(a)   = \sum_{x\in L} \mu( \lambda x^s + ax)
             = \sum_{x\in L} \mu( \lambda y^s x^s + axy )
             = \tfr f( a\lambda^{-t} ).
\end{equation}

Hence, one of the specifics of power permutations 
among the permutations  of $L$ is that the spectrum 
of $\lambda f$ does not depend on  $\lambda\in\grmul L$.

We conclude this section by giving a divisibility
result. Recall that a function $f$ defined over a field $L$ of characteristic 2 is said to be  almost perfect 
nonlinear if for all $u\in\grmul L$ the derivative
$x\mapsto f(x+u)+f(x)$ is two-to-one. It is for example the case 
of $f(x)=x^3$ over any field $L$ and of $f(x)=x^{-1}$  when $[L:\fd]$ is odd.

\begin{proposition} 
Let $f$ be a power permutation over a field $L$ of characteristic two and
cardinal $q\not\equiv 2,4\mod 5$.
If $f$ is almost perfect nonlinear  
then there exists $a\in\grmul L$ such that $\tfr f(a)\equiv 0\mod 5$.
\end{proposition}
\begin{proof}
It is well-known (see \cite{LINK}) that an APN function $f$
satisfies
\begin{equation}
\label{APN}
      \sum_{\lambda\in\grmul L} \sum_{a\in L} \tfrl f(a)^4= 2q^3(q-1).
\end{equation}
Since the spectrum of $\lambda f$ does not depend on $\lambda\in\grmul L$, it implies that:
\begin{equation}
\label{APNBIS}
      \sum_{a\in L} \tfrl f(a)^4= 2q^3.
\end{equation}

Assuming $\Det( f )\not\equiv 0\mod 5$,
we get the congruence
$
             q-1 \equiv  2 q^3 \mod 5
$ implying $q\equiv 2, 4\mod 5$.
\end{proof}

\section{hyperplane section}

The key point of view of this note is to consider
the number, say $N_n(u,v)$, of solutions in $L^n$ 
of the system 

\begin{equation}
\label{SYS}
\left\{ 
 \begin{array}{rcccccccc}
             u &= &x_1    &+ &x_2    &+&\ldots &+ &x_{n} \\
             v &= &f(x_1) &+ &f(x_2) &+&\ldots &+ &f(x_{n}).
   \end{array}
\right.
\end{equation}

Using characters counting principle, we can write:

\begin{align*}
q^2  N_n(u, v)  &=\sum_{x_1,x_2,\ldots,x_n} \sum_{\beta\in L}\sum_{\alpha\in L}
\mu_\beta(\sum_{i}f(x_i) + v) \mu_\alpha( \sum_i x_i+ u)\\
   &=\sum_{\beta}\sum_{\alpha}
    \big(\sum_{y}\mu( \beta f(y)+\alpha  y))\big)^n\mu(\alpha u+\beta v)\\
   &=\sum_{\beta}\sum_{\alpha}
    \tfr {f_\beta}(\alpha)^n \mu(\alpha u+\beta v)\\
   &=\sum_{\alpha}
    \tfr { 1 }(\alpha)^n \mu(\alpha u)
    + \sum_{\beta\not=0}\sum_{\alpha}
    \tfr {f_\beta}(\alpha)^n \mu(\alpha u+\beta v)\\
   &= q^n
     + \sum_{\alpha\not=0}\sum_{\beta\not=0}
    \tfr {f_\beta}(\alpha)^n \mu(\alpha u+\beta v)\\
\end{align*}

\begin{lemma}
\label{ELL}
Assuming the Fourier coefficients of $\lambda f$, $\lambda\in L$,
are integers.
Let $\ell$ be a prime such that 
$\prod_{\lambda\in\grmul L} \Det( \lambda f )\not\equiv 0\mod\ell$.
Then

\begin{equation*}
q^2 N_{\ell - 1}(u,v) \equiv 1 +
     (q\delta_0(u)-1) (q\delta_0(v) -1)  \mod{\ell}
\end{equation*}
where $\delta_a(b)$ is equal to $1$ if $b=a$ and $0$ otherwise.
\end{lemma}
\begin{proof}
By the Fermat's little Theorem, we have the 
congruence

$$
{\tfrl f(a)}^{\ell-1}  \equiv 1 - \delta_0(a) \mod \ell.
$$
Hence

\begin{align*}
q^2 N_{\ell - 1}(u,v) &= q^{\ell -1}
     + \sum_{\alpha\not=0}\sum_{\beta\not=0}
    \tfr {f_\beta}(\alpha)^{\ell -1} \mu(\alpha u+\beta v)\\
& \equiv 1
     + \sum_{\alpha\not=0}\sum_{\beta\not=0}
    \mu(\alpha u+\beta v) \mod{\ell}\\
\end{align*}
and we conclude remarking that $\sum_{\alpha\in\grmul L} \mu(\alpha u) = q\delta_0(u) - 1$.
\end{proof}

\section{Divisibility of Fourier coefficients}

\label{secmode}

In \cite{TOR}, it is proved that for the exponents 
$s\equiv 1\mod{(p-1)}$, the Fourier coefficients are
multiple of $p$. In this section, we are
interested in divisibility properties modulo a 
prime $\ell\not=p$.

Assuming that the Fourier coefficients of a mapping $f$,
not necessary a power function,
are rational integers, we can see that if  3 does not divide
$\Det( f )$ then we have necessarily $q\equiv 2\mod 3$.  Indeed, using 
Parseval relation, we can write

$$
1\equiv q^2 = \sum_{a\in L } \vert\tfr f(a)\vert^2 
 = \sum_{a\in L } \tfr f(a) \equiv q-1\mod 3.
$$

\begin{theorem} 
\label{DIV}
Let $f$ be the power function of exponent $s$.
If $s=1 \mod{(p-1)}$ is coprime with $q-1$ then 
$\Det(f)\equiv 0 \mod 3$.
\end{theorem}

\begin{proof}
Suppose that $\Det(f)\not\equiv 0\mod 3$.
Applying Lemma \ref{ELL} with $\ell=3$, 
we get

\begin{equation}
\label{THIS}
\forall u\in\grmul L,\quad \forall v\in\grmul L,\qquad N_{2}(u,v) \not\equiv 0 \mod \ell.
\end{equation}

To complete the proof we prove the existence of a pair $(u,v)$
of nonzero elements such that $N_2(u,v)=0$. Let us fix $u=1$,
the $v$'s such that $N_2(1,v)>0$ are in the image of $L$ by the
 mapping $x\mapsto (1-x)^s + x^s$, if $x$ is a preimage of $v$
then $1-x$ is an other one except if $p=2$ and $v=2(1/2)^s$. Thus, if
$q>3$, there exists $v\in\grmul L$ without preimage i.e.
$N_2(1,v) = 0$.

\end{proof}

\begin{proposition} 
Let $f$ be a power permutation of exponent $s \equiv 1 \mod{(p-1)}$. 
If $[L:\fp]$ is a power of a prime $\ell$ and $p\not\equiv 2\mod\ell$ 
then  $\Det( f ) \equiv 0 \mod{\ell}$.
\end{proposition}
\begin{proof}
The Frobenius automorphism acts on the solutions 
of the system (\ref{SYS}) with $u=0$,  $v=1$. Since $s\equiv 1\mod{(p-1)}$,
the system has no $\fp$-solutions, thus  $N_{\ell -1} (0,1)\equiv 0\mod {\ell}$. 
On the other hand, by Lemma \ref{ELL}, if $\Det(f) \not\equiv 0 \mod{\ell}$ then
\begin{equation*}
q^2 N_{\ell - 1}(0,1) \equiv 2 - q  \equiv 2 - p \mod{\ell}.
\end{equation*}

\end{proof}

\bibliographystyle{plain}
\nocite{*}
\bibliography{Aubry_Langevin}

\end{document}